\documentclass[12pt,a4paper]{amsart}
\usepackage{mathrsfs}
\usepackage{amssymb,amsmath,amsthm,color}
\usepackage{graphicx,mcite}
\usepackage{hyperref}
\usepackage{url}
\usepackage{setspace}
\newcommand{\loc}{\mbox{loc}}
\newcommand{\Ann}{\mbox{Ann}}
\newtheorem{theorem}{Theorem}[section]
\newtheorem{lemma}[theorem]{Lemma}
\newtheorem{proposition}[theorem]{Proposition}
\newtheorem{corollary}[theorem]{Corollary}

\theoremstyle{definition}

\newtheorem{remark}[theorem]{Remark}

\numberwithin{equation}{section}

\begin{document}

\title[Weighted Twisted Spherical Means]
{Sets of injectivity for weighted twisted spherical means and support theorems}

\author{Rajesh K. Srivastava}
\address{School of Mathematics, Harish-Chandra Research Institute, Allahabad, India 211019.}
\email{rksri@hri.res.in}


\date{\today}



\begin{abstract}
We prove that the spheres centered at origin are sets of injectivity for certain
weighted twisted spherical means on $\mathbb C^n.$ We also prove an analogue of
Helgason's support theorem for weighted Euclidean and twisted spherical means.
\end{abstract}

\maketitle

\section{Introduction}\label{section1}

In this article, we show that the spheres $S_R(o)=\{z\in\mathbb C^n: |z|=R\}$ are
sets of injectivity for the weighted twisted spherical means (WTSM) for a suitable class
of functions on $\mathbb C^n$. The weights here are spherical harmonics on $S^{2n-1}.$
In general, the question of set of injectivity for the twisted spherical means (TSM)
with real analytic weight is still open. We would like to refer to \cite{NRR}, for
some results on the sets of injectivity for the spherical means with real analytic
weights in the Euclidean setup.

\smallskip

Our main result, Theorem \ref{Cth2} is a natural generalization of a result by
Thangavelu et al. \cite{NT}, where it has been proved that the spheres $S_R(o)$'s
are sets of injectivity for the TSM on $\mathbb C^n.$
The twisted spherical mean arises in the study of spherical mean on the
Heisenberg group $\mathbb H^n=\mathbb C^n\times\mathbb R.$ These result can also be
interpreted for the weighted spherical means on the  Heisenberg group. The set
$S=\{(z,t):~|z|=R,~t\in\mathbb R\}\subset\mathbb H^n$ is a set of injectivity for the
weighted spherical means on $\mathbb H^n$ defined by (\ref{exp3}).

\smallskip

In a fundamental result, Helgason proved  a support theorem for continuous function
having vanishing spherical means over a family of spheres, sitting in the exterior of
a ball. That is, if $f$ is a continuous function on $\mathbb R^n,(n\geq2)$ such that
$|x|^kf(x)$ is bounded for each non-negative integer $k,$ then $f$ is supported in ball
$B_r(o)$ if and only if $f\ast\mu_s(x)=0, \forall x\in\mathbb R^n$ and $\forall s>|x|+r,$
(see \cite{H}).
In a recent work \cite{NT2}, Thangavelu and  Narayanan prove a support theorem,
for the TSM for certain subspace of Schwartz class functions on $\mathbb C^n.$
In our previous work \cite{RS}, we have given an exact analogue of Helgason's support
theorem for the TSM on $\mathbb C^n ~(n\geq2).$ For $n=1,$ we have proved a surprisingly
stronger result where we do not need any decay condition. This result has no analogue in
the Euclidean set up. In Theorem \ref{Cth3}, we generalize our idea of support theorem
for the TSM to the WTSM.
At the end, we revisit Euclidean spherical means and prove
Theorem \ref{Cth4}, which is an analogue of Helgason's support theorem for the weighted
spherical means. For some results on support theorem with real analytic weight, in
non-Euclidean set up, we refer to Quinto's works \cite{Q1,Q2,Q3}.

\smallskip

Let $\mu_r$ be the normalized surface measure on sphere $S_r(x).$ Let
$\mathscr F\subseteq L^1_{\loc}(\mathbb R^n).$ We say that $S\subseteq\mathbb R^n$
is a set of injectivity for the spherical means for $\mathscr F$ if for
$f\in\mathscr  F$ with $f\ast\mu_r(x)=0, \forall r>0$ and $\forall x\in S,$
implies $f=0$ a.e.

\smallskip

The results on sets of injectivity differ in the choice of sets and the class of functions
considered. The following result by Agranovsky et al. \cite{ABK} partially describe the sets
of injectivity in $\mathbb R^n.$ The boundary of bounded domain in $\mathbb R^n~(n\geq2)$
is set of injectivity for the spherical means on $L^p(\mathbb R^n),$
$~1\leq p\leq\frac{2n}{n-1}.$ For $p>\frac{2n}{n-1},$ unit sphere $S^{n-1}$ is an example
of non-injectivity set in $\mathbb R^n.$

\smallskip

The range for $p$ in the above result is optimal. That can be seen as follows.
For $\lambda>0,$ define the radial function $\varphi_\lambda$ on $\mathbb R^n$ by
\[\varphi_\lambda(x)=c_n(\lambda|x|)^{-\frac{n}{2}+1}J_{\frac{n}{2}-1}(\lambda|x|),\]
where $J_{\frac{n}{2}-1}$ is the Bessel function of order ${\frac{n}{2}-1}$ and
$c_n$ is the constant such that $\varphi_\lambda(o)=1$. Then the spherical means of
$\varphi_\lambda$ satisfy the relation
\[\varphi_\lambda \ast \mu_r(x)=\varphi_\lambda(r)\varphi_\lambda(x).\]
This shows that if  $\lambda R$ is zero of Bessel function
$J_{\frac{n}{2}-1}$ then $\varphi_\lambda \ast \mu_r(x)=0$ on sphere
$S_R(o)$ and for all $r>0$. Since $\varphi_\lambda\in L^p(\mathbb R^n)$
if and only if $p>2n/(n-1),$ it follows that spheres are not sets of
injectivity for spherical means for $L^p$ for $p>2n/(n-1)$. In a recent
result of Narayanan et al. \cite{NRR}, it has been shown that the boundary
of a bounded domain in $\mathbb R^n$ is a set of injectivity for the
weighted spherical means for $L^p(\mathbb R^n),$ with $1\leq p\leq\frac{2n}{n-1}.$

\bigskip

Next, we come up  with twisted spherical means which arises in the study
of spherical means on Heisenberg group. The group $\mathbb H^n$ as a manifold,
is $\mathbb C^n \times\mathbb R$ with the group law
\[(z, t)(w, s)=(z+w,t+s+\frac{1}{2}\text{Im}(z.\bar{w})),~z,w\in\mathbb C^n
\text{ and }t,s\in\mathbb R.\]

The spherical means of a function $f$ in $L^1(\mathbb H^n)$ are defined by
\begin{equation}\label{exp2}
f\ast\mu_s(z, t)=\int_{|w|=s}~f((z,t)(-w,0))~d\mu_s(w).
\end{equation}
Thus the spherical means can be thought of as convolution operators.
An important technique in many problem on $\mathbb H^n$ is to take
partial Fourier transform in the $t$-variable to reduce matters to
$\mathbb C^n$. This technique works very well with convolution
operator on $\mathbb H^n$  and we will make use of it to analyze
spherical means on $\mathbb H^n$. Let
\[f^\lambda(z)= \int_\mathbb R f(z,t)e^{i \lambda t} dt\]
 be the inverse Fourier transform of $f$ in the $t$-variable.
Then a simple calculation shows that
\begin{eqnarray*}
(f \ast \mu_s)^\lambda(z)&=&\int_{-\infty}^{~\infty}~f \ast \mu_s(z,t)e^{i\lambda t} dt\\
&=&\int_{|w| = s}~f^\lambda (z-w)e^{\frac{i\lambda}{2} \text{Im}(z.\bar{w})}~d\mu_s(w)\\
&=&f^\lambda\times_\lambda\mu_s(z),
\end{eqnarray*}
where $\mu_s$ is now being thought of as normalized surface measure
on the sphere $S_s(o)=\{ z \in \mathbb C^n: |z|=s\}$ in $\mathbb
C^n.$
Thus the spherical mean $f\ast \mu_s$ on the Heisenberg group can be
studied using the $\lambda$-twisted spherical mean $f^\lambda
\times_\lambda \mu_s$ on $\mathbb C^n.$ For $\lambda \neq 0,$
a further scaling argument shows that it is enough to study these
means for the case of $\lambda = 1.$

Let $\mathcal F\subseteq L^1_{\loc}(\mathbb C^n).$  We say $S\subseteq\mathbb C^n$
is a set of injectivity for twisted spherical means for $\mathcal F$ if for
$f\in\mathcal F$ with $f\times\mu_r(z)=0, \forall r>0$ and $\forall z\in S,$
implies $f=0$ a.e. on $\mathbb C^n.$

As in the Euclidean case, it would be natural to ask if the
boundaries of bounded domains in $\mathbb C^n$  continue to be sets
of injectivity for $L^p$ spaces for the twisted spherical means.
However, this is no longer true as can be seen by considering the
Laguerre functions $\varphi^{n-1}_k, k\in\mathbb Z_+,$ given by
$\varphi^{n-1}_k(z)=L_k^{n-1}\left(\frac{1}{2}|z|^2\right)e^{-\frac{1}{4}|z|^2},$
where $L_k^{n-1}$'s are the Laguerre polynomials of degree $k$ and
type $n-1.$ These functions satisfy the functional relations
\begin{equation}\label{Cexp22}
\varphi^{n-1}_k\times\mu_r(z)=\frac{k!(n-1)!}{(n+k-1)!}\varphi^{n-1}_k(r)\varphi^{n-1}_k(z),~k\in\mathbb Z_+.
\end{equation}
For $k=0, \varphi^{n-1}_0(z)=e^{-\frac{1}{4}|z|^2},$ which is never
zero. Otherwise, if $\frac{1}{2}R^2$ is a zero of $L_k^{n-1}$  for
$k=1,2,\ldots,$ then $\varphi^{n-1}_k\times\mu_r(z)=0$ on sphere
$S_R(o)$ for all $r>0.$ Since $\varphi^{n-1}_k$ are in Schwartz
class, it follows that spheres, and hence boundaries of bounded
domains are not sets of injectivity for $L^p(\mathbb C^n)$ for any
$p~,1\leq p\leq\infty.$ As $e^{\frac{1}{4}|z|^2}\varphi^{n-1}_k,~
k=1,2, \ldots,$ does not belong to $L^p(\mathbb C^n)$ for  $1\leq
p\leq\infty,$ it would be interesting to know if
 boundaries of bounded
domains in $\mathbb C^n$ are sets of injectivity for the class  of
functions  $f$ such that $f(z)e^{\frac{1}{4}|z|^2}\in L^p(\mathbb
C^n)$ for some $p,1\leq p\leq\infty$. In \cite{AR} the authors
answer this for a yet smaller function space. The boundary of a bounded
domain in $\mathbb C^n$ is set of injectivity for function $f$ with
$f(z)e^{(\frac{1}{4}+\epsilon)|z|^2}\in L^p(\mathbb C^n)$ for some
$\epsilon>0$ and $1\leq p\leq\infty$. In the light of the above discussion
an optimal result would be proving this result for $\epsilon=0$. This in
general is an open problem, but in the special case of $\Gamma=S^{2n-1},$
the result has been established by Narayanan and Thangavelu \cite{NT}.

\begin{theorem} \label{Cth9}\emph{\cite{NT}}
Let $f$ be a function on $\mathbb C^n$ such that $e^{\frac{1}{4}|z|^2}f(z)\in
L^p(\mathbb C^n),$ for $1\leq p\leq \infty$. If $f\times\mu_r(z)=0$ on sphere
$S_R(o)$ and for all $r>0$, then $f=0$ a.e. on $\mathbb C^n.$
\end{theorem}

\begin{remark}\label{rk2}
For $\eta\in\mathbb C^n,$ define the left twisted translate by
\[\tau_\eta f(\xi)=f(\xi-\eta)e^{\frac{i}{2}\text{Im}(\eta.\bar\xi)}.\]
Then $\tau_{\eta}(f\times\mu_r)=\tau_{\eta}f\times\mu_r.$ Since the
function space considered as in the above Theorem \ref{Cth9} is not twisted
translation invariant, it follows that a sphere centered off origin is not
set of injectivity for the TSM on $\mathbb C^n.$
\end{remark}

Our aim is to consider some special weighted twisted spherical means
and prove that Theorem \ref{Cth9} can be extended for those means.
For this, let $\mathbb Z_+$ denote the set of all non negative-integers. For
$s,t\in\mathbb Z_+$, let $P_{s,t}$ denote the space of all
polynomials $P$ in $z$ and $\bar{z}$ of the form
\[P(z)=\sum_{|\alpha|=s}\sum_{|\beta|=t}c_{\alpha\beta} z^\alpha\bar{z}^\beta.\]
Let $H_{s,t}=\{P\in P_{s,t}:\Delta P=0\},$ where $\Delta$ is the
standard Laplacian on $\mathbb C^n.$ Let $\{P\in P_{s,t}^j:~1\leq j\leq d(s,t)=\dim H_{s,t}\}$
be an orthonormal basis of $H_{s,t}$ and $d\nu_{r,j}=P_{st}^jd\mu_r.$ Then
$d\nu_{r,j}$ is a signed measure on the sphere $S_r(o)$ in $\mathbb C^n.$
As similar to (\ref{exp2}), we can define the weighted spherical means of
a function $f\in L^1(\mathbb H^n)$ by
\begin{equation}\label{exp3}
f\ast\nu_{r,j}(z, t)=\int_{|w|=s}~f((z,t)(-w,0))~P_{st}^j(w)d\mu_s(w).
\end{equation}
By taking the inverse Fourier transform in $t$ variable at $\lambda=1,$
we can write
\[f\times\nu_{r,j}(z)=\int_{S_r(o)}f(z-w)e^{\frac{i}{2}\text{Im}(z.\bar{w})}P_{st}^j(w)d\mu_r(w).\]
We call $f\times\nu_{r,j}$ the weighted twisted spherical mean (WTSM)
of function $f\in L_{\loc}(\mathbb C^n).$ We prove the following result
for the injectivity of the WTSM.

\begin{theorem}\label{Cth2}
Let $f$ be a function on $\mathbb C^n$ such that $e^{\frac{1}{4}|z|^2}f(z)\in L^p(\mathbb C^n),$
$1\leq p<\infty$. If $f\times\nu_{r,j}(z)=0$ on sphere $S_R(o)$, $\forall ~r>0$ and $\forall
~j,1\leq j\leq d(s,t)$, then $f=0$ a.e. on $\mathbb C^n.$
\end{theorem}
For $p=\infty,$ Theorem \ref{Cth2} does not hold as can be seen in Remark \ref{Crk1}.
Further, we prove a support theorem for the weighted twisted spherical means.
\begin{theorem}\label{Cth3}
Let $f$ be a smooth function on $\mathbb C^n$ such that for each
non-negative integer $k$, $|z|^{k}|f(z)|\leq C_k~e^{-\frac{1}{4}|z|^2}$.
Let $f\times\nu_{r,j}(z)=0,$ for all $z\in \mathbb C^n$ and $ r>|z|+B $
and for all $j,1\leq j\leq d(s,t)$. Then $f=0,$ whenever $|z|>B$.
\end{theorem}
In the end, we revisit Euclidean spherical means and prove the support
Theorem \ref{Cth4} for the weighted spherical means. For $k\in \mathbb Z^+$,
let $P_k$ denote the space of all homogeneous polynomials $P$ of degree $k.$
Let $H_k = \{ P \in P_k~:~\Delta P =0 \}.$ The elements of $H_k$ are called
the solid spherical harmonics of degree $k.$  Let $\{P_{kj}:~1\leq j\leq d_k=\dim H_k\}$
be an orthonormal basis for $H_k.$ Define the weighted spherical mean of
function $f\in L_{\loc}^1(\mathbb R^n)$ by
\[f\ast\mu_{r,j}^k(x)=\int_{S_r(o)}f(x+y)P_{kj}(y)d\mu_r(y).\]
\begin{theorem}\label{Cth4}
Let $f$ be a smooth function on $\mathbb R^n$ such that $|x|^mf(x)$ is
bounded for each $m\in\mathbb Z_+$. Let $f\ast\mu_{r,j}^k(x)=0,$ for all
$x\in\mathbb R^n,~r>|x|+B$ and for all $j,1\leq j\leq d_k$. Then $f=0$
whenever $|x|>B$.
\end{theorem}

\section{Preliminaries}\label{section2}
 We need the following basic facts from the theory of bigraded
spherical harmonics (see \cite{T}, p.62 for details). We shall use
the notation $K=U(n)$ and $M=U(n-1).$ Then $S^{2n-1}\cong K/M$ under
the map $kM\rightarrow k.e_n,$ $k\in U(n)$ and $e_n=(0,\ldots
,1)\in \mathbb C^n.$ Let $\hat{K}_M$ denote the set of all
equivalence classes of irreducible unitary representations of $K$
which have a nonzero $M$-fixed vector. It is known that for each
representation in $\hat{K}_M$ has a unique nonzero $M$-fixed vector,
up to a scalar multiple.

For a $\delta \in \hat{K}_M,$ which is realized on $V_{\delta},$ let
$\{e_1,\ldots ,e_{d(\delta)}\}$ be an orthonormal basis of
$V_{\delta}$ with $e_1$ as the $M$-fixed vector. Let
$t_{ij}^{\delta}(k)=\langle e_i,\delta (k)e_j \rangle ,$ $k\in K$
and $\langle , \rangle$ stand for the innerproduct on $V_{\delta}.$
By Peter-Weyl theorem, it follows that $\{\sqrt{d(\delta
)}t_{j1}^{\delta}:1\leq j\leq d(\delta ),\delta\in\hat{K}_M\}$ is an
orthonormal basis of $L^2(K/M)$ (see \cite{T}, p.14 for details).
Define $Y_j^{\delta} (\omega )=\sqrt{d(\delta )}t_{j1}^{\delta}(k),$
where $\omega =k.e_n\in S^{2n-1},$ $k \in K.$ It then follows that
$\{Y_j^{\delta}:1\leq j\leq d(\delta ),\delta\in \hat{K}_M, \}$
forms an orthonormal basis for $L^2(S^{2n-1}).$

For our purpose, we need a concrete realization of the
representations in $\hat{K}_M,$ which can be done in the following
way. See \cite{Ru}, p.253, for details.

For $p,q\in\mathbb Z_+$, let $P_{p,q}$ denote the space of all
polynomials $P$ in $z$ and $\bar{z}$ of the form
\[P(z)=\sum_{|\alpha|=p}\sum_{|\beta|=q}c_{\alpha\beta} z^\alpha
\bar{z}^\beta.\]
Let $H_{p,q}=\{P\in P_{p,q}:\Delta P=0\}.$ The elements of $H_{p,q}$ are
called the bigraded solid harmonics on $\mathbb C^n.$ The group  $K$ acts
on $H_{p,q}$ in a natural way. It is easy to see that the space $H_{p,q}$
is $K$-invariant. Let $\pi_{p,q}$ denote the corresponding representation
of $K$ on $H_{p,q}.$ Then representations in $\hat{K}_M$ can be identified,
up to unitary equivalence, with the collection $\{\pi_{p,q}: p,q \in \mathbb Z_+\}.$

Define the bigraded spherical harmonic by $Y_j^{p,q}(\omega)=\sqrt{d(p,q )}t_{j1}^{p,q}(k).$
Then $\{Y_j^{p,q}:1\leq j\leq d(p,q)~\text{and}~p,q \in \mathbb Z_+ \}$ forms an orthonormal basis for
$L^2(S^{2n-1}).$ Therefore for a continuous function $f$ on $\mathbb C^n,$ writing
$z=\rho \,\omega,$ where $\rho>0$ and $\omega \in S^{2n-1},$ we can expand the function
$f$ in terms of spherical harmonics as
\begin{equation}\label{Bexp4}
f (\rho\omega) = \sum_{p,q\geq0} \sum_{j=1}^{d(p,q)} a_j^{p,q}(\rho)
Y_j^{p,q}(\omega).
\end{equation}
The functions $ a_j^{p,q} $ are called the spherical harmonic
coefficients of the function $f$. The $(p,q)^{th}$ spherical
harmonic projection, $\Pi_{p,q}(f)$ of the function $f$ is then
defined as
\begin{equation}\label{Bexp5}
\Pi_{p,q}(f)(\rho,\omega)= \sum_{j=1}^{d(p,q)} a_j^{p,q}(\rho)
Y_j^{p,q}(\omega).
\end{equation}
We will replace the spherical harmonic $Y_j^{p,q}(\omega)$ on the
sphere by the solid harmonic
$P_j^{p,q}(z)=|z|^{p+q}Y_j^{p,q}(\frac{z}{|z|})$ on $\mathbb C^n$
and accordingly for a function $f.$ Define $\tilde{a}_j^{p,q}(\rho)=
\rho^{-(p+q)} a_j^{p,q}(\rho),$ where $a_j^{p,q}$ are defined by
Equation \ref{Bexp4}. We shall continue to call the functions
$\tilde{a}_j^{p,q}$ the spherical harmonic coefficients of $f.$

\bigskip

In the proof of Theorem \ref{Cth2}, we also need an expansion of functions
on $\mathbb C^n$ in terms of Laguerre functions $\varphi_k^{n-1}$'s. Let
$f\in L^2(\mathbb C^n).$ Then the special Hermite expansion for $f$ is
given by
\begin{equation}\label{Cexp16}
f(z)=(2\pi)^{-n}\sum_{k=0}^\infty f\times\varphi_k^{n-1}(z).
\end{equation}
For radial functions, this expansion further simplifies as can be seen from the following lemma.
\begin{lemma} \label{lemma1C}\emph{\cite{T}}
Let  $f$ be a radial function in $L^2(\mathbb C^n)$. Then
\[f=\sum_{k=0}^\infty B^n_k\left\langle f,\varphi^{n-1}_k\right\rangle\varphi^{n-1}_k,
~\text{where}~B^n_k=\frac{k!(n-1)!}{(n+k-1)!}.\]
\end{lemma}
We would also need the following Hecke-Bochner identities for
the spectral projections $f\times\varphi^{n-1}_k,$  (see \cite{T}, p.70).

\begin{lemma} \label{lemma4C}\emph{\cite{T}}
Let ~$\tilde aP\in L^2(\mathbb C^n),$ where $\tilde a$ is radial and $P\in H_{p,q}$.
Then
\[(\tilde aP)\times\varphi_k^{n-1}(z)=(2\pi)^{-n}P(z)~\tilde a\times\varphi_{k-p}^{n+p+q-1}(z),\]
if ~$k\geq p$ and ~$0$ otherwise. The convolution in the right hand side is on
the space $~\mathbb C^{n+p+q}.$
\end{lemma}
Using the Hecke-Bochner identities, a weighted functional equation for spherical function
$\varphi_k^{n-1}$ has been proved in \cite{T}, p.98.
\begin{lemma}\label{lemma3C}\emph{\cite{T}}
For $z\in\mathbb C^n,$ let $P\in H_{p,q}$ and $d\nu_r=Pd\mu_r$.
Then
\[\varphi_k^{n-1}\times\nu_r(z)=
(2\pi)^{-n}C(n,p,q)r^{2(p+q)}\varphi_{k-q}^{n+p+q-1}(r)P(z)\varphi_{k-q}^{n+p+q-1}(z),\]
if $k\geq q$ and ~$0$ otherwise.
\end{lemma}
\begin{remark}\label{Crk1}
From Lemma \ref{lemma3C}, it can  be seen that Theorem \ref{Cth2} does not hold for
$p=\infty.$ For instance, take $P\in H_{0,1}$ and let
$d\nu=Pd\mu_r$. Then $\varphi_0^{n-1}\times\nu(z)=0$, where
$\varphi_0^{n-1}(z)=e^{-\frac{1}{4}|z|^2}.$
\end{remark}

\section{Injectivity of the Weighted Twisted Spherical Means}\label{section3}
In this section, we prove that the spheres are sets of injectivity for the
weighted twisted spherical means on $\mathbb C^n$.
Let
\begin{equation}\label{Cexp2}
f_{lm}(z)=d(s,t)\int_{U(n)}f(\sigma^{-1}z)t_{lm}^{s,t}(\sigma)d\sigma
\end{equation}
for $1 \leq l,m\leq d(s,t).$
\begin{lemma} \label{lemma7C}
Let $f$ be a continuous function on $\mathbb C^n$. Suppose $f\times\nu_{r,j}(z)=0$
on sphere $S_R(o),$ for all $j,~1\leq j\leq d(p,q)$ and for all $r>0$. Then
$f_{lm}\times\nu_{r,j}(z)=0,$ on $S_R(o)$, whenever $1\leq l,m\leq d(s,t),~
1\leq j\leq d(p,q)$ and $r>0$.
\end{lemma}
\begin{proof}
We have
\[f_{lm}\times\nu_{r,j}(z)=d(s,t)\int_{S_r(o)}\int_{U(n)}
f(\sigma^{-1}(z-w))e^{\frac{i}{2}\text{Im}(z.\bar{w})}t_{lm}^{s,t}(\sigma)P_{s,t}^j(w)d\sigma d\mu_r(w).\]
Since the space $H_{p,q}$ is $U(n)$-invariant, therefore the function
$P_{s,t}^j(\sigma^{-1}w)$ is linear combination of polynomials in $H_{p,q}$.
By hypothesis, it follows that
\[\int_{U(n)}t_{lm}^{s,t}(\sigma)
\int_{S_r(o)}f(\sigma^{-1}z-w)e^{\frac{i}{2}\text{Im}(\sigma^{-1}z.\bar{w})}P_1^j(\sigma w)d\mu_r(w)d\sigma=0.\]
\end{proof}

\begin{remark}\label{Crk2}
In view of Lemma \ref{lemma7C}, it is enough to work with the function of
type $f(z)=\tilde a(|z|)P_{s,t}(z)$ and measure $d\nu_r=z_1^p\bar z_2^qd\mu_r$
for the proof of Theorem \ref{Cth2}. We therefore drop the index $j$ and write
$P_1(z)=z_1^p\bar z_2^q$ and $d\nu_r=P_1d\mu_r.$
\end{remark}
We need the following result of Filaseta and Lam \cite{FL}, about the
irreducibility of Laguerre polynomials. Define the
Laguerre polynomials by
\[L^\alpha_k(x)=\sum_{i=0}^k(-1)^i\binom{\alpha+k}{k-i}\frac{x^i}{i!},\]
$\text{ where } k\in \mathbb Z_{+} \text{ and } \alpha\in\mathbb C.$

\begin{theorem}\emph{\cite{FL}}\label{Cth10} Let $\alpha$ be a rational number,
which is not a negative integer. Then for all but finitely many $k\in\mathbb Z_+$,
the polynomial $L_k^\alpha(x)$ is irreducible over the rationals.
\end{theorem}
Using Theorem \ref{Cth10}, we obtain the following corollary about the zeros of Laguerre
polynomials.
\begin{corollary}\label{cor1}
Let $k\in\mathbb Z_{+}.$ Then for all but finitely many $k$, the Laguerre polynomials
$L^{n-1}_k(x)$'s have distinct zeros over the reals.
\end{corollary}
\begin{proof}
 By Theorem \ref{Cth10}, there exists $k_o\in\mathbb Z_+$ such that $L_k^{n-1}$'s are
irreducible over $\mathbb Q$ whenever  $k\geq k_o.$ Therefore, we can find polynomials
$P_1, P_2\in\mathbb Q[x]$ such that $P_1L_{k_1}^{n-1}+ P_2L_{k_2}^{n-1}=1,$ over $\mathbb Q$
with $k_1,k_2\geq k_o.$ Since this identity continue to hold on $\mathbb R,$ it follows that
$L_{k_1}^{n-1}$ and $L_{k_2}^{n-1}$ have no common zero over $\mathbb R.$
\end{proof}
In the proof of Theorem \ref{Cth2}, we use the following right invariant differential operators
for twisted convolution:
\[\tilde A_j=\frac{\partial}{\partial z_j}+\frac{1}{4}\bar{z}_j ~\mbox{and}~\tilde A_j^*=
\frac{\partial}{\partial\bar{z}_j}-\frac{1}{4}z_j;~j=1,2,\ldots,n.\]
In addition, we have the left invariant differential operators
\[\tilde Z_j=\frac{\partial}{\partial z_j}-\frac{1}{4}\bar{z}_j~\mbox{and}
~\tilde Z_j^*=\frac{\partial}{\partial\bar{z}_j}+\frac{1}{4}z_j;~j=1,2,\ldots,n\] for twisted
convolution.
Let $P$ be a non-commutative homogeneous harmonic polynomial on $\mathbb C^n$
with expression
\[P(z)=\sum_{|\alpha|=p}\sum_{|\beta|=q}c_{\alpha\beta}z^\alpha\bar{z}^\beta,\]
Using the result of Geller \cite{Ge}, about Weyl correspondence of the
spherical harmonics, the operator analogue of $P(z),$ accordingly
the left and right invariant vector fields can be expressed as
\[P(\widetilde Z)=\sum_{|\alpha|=p}\sum_{|\beta|=q}c_{\alpha\beta}{\widetilde {Z^*}}^\alpha\widetilde Z^\beta
\text{ and }P(\tilde A)=\sum_{|\alpha|=p}\sum_{|\beta|=q}c_{\alpha\beta}{\tilde {A^*}}^\alpha\tilde A^\beta.\]
In order to prove Theorem \ref{Cth2}, We need to prove the following lemma.
\begin{lemma}\label{lemma5C}
For  $P_1(z)=z_1^p\bar z_2^q\in H_{p,q}$ we have
\begin{equation}\label{Cexp3}
P_1(\tilde A)\varphi_k^{n-1}(z)=\bar P_1(\widetilde Z)\varphi_k^{n-1}(z)=
(-2)^{-p-q}P_1(z)\varphi_{k-q}^{n+p+q-1}(z),
\end{equation}
if ~$k\geq q$ and ~$0$ otherwise.
\end{lemma}

\begin{proof}
We have
\[\tilde A_1^*\varphi_{k}^{n-1}(z)=\left(\frac{\partial}{\partial\bar z_1}-\frac{1}{4}z_1\right)\varphi_{k}^{n-1}(z)\]
For $z\in\mathbb C^n$, let $z.\bar z=2t$. By chain rule $\frac{\partial}{\partial\bar z_1}=
\frac{1}{2} z_1\frac{\partial}{\partial t}.$
Therefore,
\begin{eqnarray*}
\tilde A_1^*\varphi_{k}^{n-1}(z)&=&\left(\frac{1}{2}z_1\frac{\partial}{\partial t}-
\frac{1}{4}z_1\right)\left(L_{k}^{n-1}(t)e^{-\frac{1}{2}t}\right)\\
&=&\frac{1}{2}z_1\left(\frac{\partial}{\partial t}L_{k}^{n-1}(t)
-\frac{1}{2}L_{k}^{n-1}(t)-\frac{1}{2}L_{k}^{n-1}(t)\right)e^{-\frac{1}{2}t}.
\end{eqnarray*}
The Laguerre polynomials satisfy
\begin{equation}\label{Cexp23}
\frac{d}{dx}L_k^n(x)=-L_{k-1}^{n+1}(x)\text{ and }L_{k-1}^{n+1}(x)+L_k^n(x)=L_k^{n+1}(x).
\end{equation}
Thus we have
$\tilde A_1^*\varphi_{k}^{n-1}(z)=-\frac{1}{2}z_1\varphi_{k}^{n}(z).$
Similarly
\begin{eqnarray*}
\tilde A_2\varphi_{k}^{n-1}(z)&=&\left(\frac{1}{2}\bar z_2\frac{\partial}{\partial t}+
\frac{1}{4}\bar z_2\right)\left(L_{k}^{n-1}(t)e^{-\frac{1}{2}t}\right)\\
&=&\frac{1}{2}\bar z_2\left(\frac{\partial}{\partial t}L_{k}^{n-1}(t)
-\frac{1}{2}L_{k}^{n-1}(t)+\frac{1}{2}L_{k}^{n-1}(t)\right)e^{-\frac{1}{2}t}\\
&=&-\frac{1}{2}\bar z_2\varphi_{k-1}^{n}(z).
\end{eqnarray*}
Therefore,
\[\tilde A_1^*\tilde A_2\varphi_{k}^{n-1}(z)=2^{-2}z_1\bar z_2\varphi_{k-1}^{n+1}(z).\]
Since the operators $\tilde A_1^*$ and $\tilde A_2$ commute with each other,
we can conclude that
\[\tilde {A_1^*}^p\tilde A_2^q\varphi_{k}^{n-1}(z)=(-2)^{-p-q}~z_1^p\bar z_2^q~\varphi_{k-q}^{n+1}(z).\]
A similar computation shows that
\[\widetilde{Z_1^*}^p\tilde Z_2^q\varphi_{k}^{n-1}(z)=(-2)^{-p-q}~z_1^p\bar z_2^q~\varphi_{k-p}^{n+1}(z).\]
\end{proof}

\begin{remark}\label{rk3}
Using the result of Geller (\cite{Ge}, Lemma 2.4), the identity (\ref{Cexp3})
can be generalized for any $P\in H_{p,q}.$ The complete proof of this identity
require some of the preliminaries about Weyl correspondence of spherical harmonic
from the work of Geller \cite{Ge} and will be presented elsewhere.
\end{remark}

\begin{lemma}\label{lemma9C}
For $\rho>0,$ write
$\widetilde D=\frac{\partial}{\partial\rho}-\frac{1}{2}\rho$ and
$\widetilde D^*=\frac{\partial}{\partial\rho}+\frac{1}{2}\rho.$
Then $\frac{1}{\rho}\widetilde D\varphi_k^{n-1}(\rho)=\varphi_k^{n}(\rho)$ and
 $\frac{1}{\rho}\widetilde D^*\varphi_k^{n-1}(\rho)=\varphi_{k-1}^{n}(\rho).$
\end{lemma}
\begin{proof}
Let $\rho^2=2t$, then $\frac{\partial}{\partial\rho}=\rho\frac{\partial}{\partial t}.$
Therefore,
\begin{eqnarray*}
\widetilde D\varphi_{k}^{n-1}(\rho)&=&\rho\left(\frac{\partial}{\partial t}-
\frac{1}{2}\right)\left(L_{k}^{n-1}(t)e^{-\frac{1}{2}t}\right)\\
&=&\rho\left(\frac{\partial}{\partial t}L_{k}^{n-1}(t)
-\frac{1}{2}L_{k}^{n-1}(t)-\frac{1}{2}L_{k}^{n-1}(t)\right)e^{-\frac{1}{2}t}.
\end{eqnarray*}
Using (\ref{Cexp23}), we have $\frac{1}{\rho}\widetilde D\varphi_k^{n-1}(\rho)=\varphi_k^{n}(\rho)$.
Similarly,
\begin{eqnarray*}
\widetilde D^*\varphi_{k}^{n-1}(\rho)&=&\rho\left(\frac{\partial}{\partial t}+
\frac{1}{2}\right)\left(L_{k}^{n-1}(t)e^{-\frac{1}{2}t}\right)\\
&=&\rho\left(\frac{\partial}{\partial t}L_{k}^{n-1}(t)
-\frac{1}{2}L_{k}^{n-1}(t)+\frac{1}{2}L_{k}^{n-1}(t)\right)e^{-\frac{1}{2}t}.
\end{eqnarray*}
Therefore $\frac{1}{\rho}\widetilde D^*\varphi_k^{n-1}(\rho)=\varphi_{k-1}^{n}(\rho)$.
\end{proof}

Suppose $f$ be a function on $\mathbb C^n$ such that $e^{\frac{1}{4}|z|^2}f(z)\in L^p(\mathbb C^n),$
for $1\leq p<\infty.$ Let $\varphi_\epsilon$ be a smooth, radial compactly supported
approximate identity on $\mathbb C^n.$ Then $f\times\varphi_\epsilon\in L^1\cap L^\infty(\mathbb C^n)$
and in particular $f\times\varphi_\epsilon\in L^2(\mathbb C^n).$ Let $d\nu_r=Pd\mu_r.$ Suppose
$f\times\nu_r(z)=0, \forall r>0$ and $\forall z\in S_R(o).$ Then by polar decomposition
$f\times P\varphi_{k-q}^{n+p+q-1}(z)=0, \forall k\geq q$ and $\forall z\in S_R(o).$
Since $\varphi_\epsilon$ is radial, we can write
\[f\times\varphi_\epsilon\times\nu_r(z)=\sum_{k\geq 0}B_k^n\left\langle\varphi_\epsilon,\varphi_k^{n-1}\right\rangle
f\times\varphi_k^{n-1}\times\nu_r(z).\]
By Lemma \ref{lemma3C}, it follows that
$f\times\varphi_\epsilon\times\nu_r(z)=0, \forall k\geq q$ and $\forall z\in S_R(o).$
Thus without loss of generality, we can assume $f\in L^2(\mathbb C^n).$
Hence to prove the Theorem \ref{Cth2}, in view of Lemma \ref{lemma7C}, it is enough to prove the
following result.

\begin{proposition}\label{Cprop6}
Let $P_{s,t}\in H_{s,t}$ and  $f=\tilde{a}P_{s,t}\in L^2(\mathbb C^n)$ be a smooth function such that
$e^{\frac{1}{4}|z|^2}f(z)\in L^p(\mathbb C^n),$ for $1\leq p<\infty$. If $f\times\nu_r(z)=0$
on $S_R(o)$ and for all $r>0$, then $f=0$ a.e.
\end{proposition}
\begin{proof}
We have \[f=(2\pi)^{-n}\sum_{k\geq0}f\times\varphi^{n-1}_k.\]Therefore
\[\sum_{k\geq 0}f\times(\varphi_k^{n-1}\times P_1\mu_r)(z)=0,\]
whenever $z\in S_R(o)$ and $r>0$. By Lemma \ref{lemma3C}, we get
\[\sum_{k\geq q}C(n,p,q)\varphi_{k-q}^{n+p+q-1}(r)f\times
P_1\varphi_{k-q}^{n+p+q-1}(z)=0,\]
for $|z|=R$ and for all $r>0$. As the functions $\left\{\varphi_{k-q}^{n+p+q-1}(r): k\geq q\right\}$
form an orthonormal basis for $L^2\left(~\mathbb R^+,r^{2(n+p+q)-1}dr \right),$
the above implies that
\[f\times P_1\varphi_{k-q}^{n+p+q-1}(z)=0, ~\forall~k\geq q \text{ and }|z|=R.\]
From Lemma \ref{lemma5C},
$P_1(\tilde A)\varphi_k^{n-1}(z)=(-2)^{-p-q}P_1(z)\varphi_{k-q}^{n+p+q-1}(z)$,
moreover $P_1(\tilde A)$ is right invariant, therefore it follows that
\[P_1(\tilde A)(\tilde{a}P_{s,t}\times \varphi_{k}^{n-1})(z)=0, ~\forall~k\geq
q \text{ and }|z|=R.\]
Using Hecke-Bochner identity (Lemma \ref{lemma4C}), we get
\[\left\langle\tilde{a},\varphi_{k-s}^{n+s+t-1}\right\rangle P_1(\tilde A)P_{s,t}\varphi_{k-s}^{n+s+t-1}(z)=0,
~\forall~k\geq\max(q,s) \text{ and }|z|=R.\]
If $\tilde {A_1^*}^p\tilde {A_2}^q(P_{s,t}\varphi_{k-s}^{n+s+t-1})(R)=0$
for some $k\geq\max(q,s)$, then a computation similar as done for $Z_j^*f$
in \cite{RS}, p.2516-17, we have
\[\tilde {A_1^*}^p\tilde {A_2}^{q-1}\left[\frac{{1}}{2\rho}\widetilde D^*\varphi_{k-s}^{\gamma-1}P_{s+1,t}+
\left\{\left(\frac{{1}}{2(\gamma-1)}\rho\widetilde D^*+1\right)\varphi_{k-s}^{\gamma-1}\right\}
\frac{\partial{P_{s,t}}}{\partial\bar{z_2}}\right]=0,\]
for $|z|=R$  and $\gamma=n+s+t$.  Since $\{P_{s,t}\rvert_{S^{2n-1}}: s, t\geq0 \}$
form an orthonormal basis for $L^2(S^{2n-1})$. An inductive process, then
gives the coefficient of highest degree polynomial $P_{p+s,q+t}$ as
\[
\left(\frac{{1}}{\rho}{\widetilde D}\right)^p\left(\frac{{1}}{\rho}{\widetilde D^*}\right)^q
\varphi_{k-s}^{\gamma-1}(R)=0.\]
Using Lemma \ref{lemma9C}, the above equation implies that
$\varphi_{k-s-q}^{\gamma+p+q-1}(R)=0.$ In view of Corollary \ref{cor1}, without
loss of generality, we can assume, the Laguerre polynomials $L_{k-s-q}^{\gamma+p+q-1}$
have distinct zeros. Hence $L_{k-s-q}^{\gamma+p+q-1}(\frac{1}{2}R^2)$ can vanish for at most
one value say $k_0\geq s+q$ of $k\geq \max(q,s)$. Therefore
$\left\langle\tilde{a},\varphi_{k-s}^{\gamma-1}\right\rangle=0,$ for $k\geq
\max(q,s)$, except for $k\neq k_0$. Hence $\tilde{a}(\rho)$ is
finite linear combination of $\varphi_{k-s}^{\gamma-1}$'s. As
$\tilde{a}$ satisfies the same decay condition as $f$, it follows
that $\tilde{a}=0$. This completes the proof.
\end{proof}

\begin{remark}\label{rk4}
In the proof of Theorem \ref{Cth2}, we have used the fact that the WTSM
$f\times\nu_{r,j}$ vanishes for each $j:~1\leq j\leq d(s,t).$ It would be
an interesting question to consider a single weight or, in general, a real
analytic weight, which we leave open for the time being.

\end{remark}

\section{Support Theorems for the Weighted Spherical Means}\label{section4}
In this section, we prove Theorem \ref{Cth3}, which is an analogue of the
author's support theorem (\cite{RS}, Therorem 1.2) for the TSM to the WTSM on
$\mathbb C^n.$ Our previous result (\cite{RS}, Theorem 1.2) is a special case
of Theorem \ref{Cth3}, when for $p=q=0.$ We would like to quote support
theorem for the case $n=1.$  In the end, we would revisit Euclidean spherical
means and indicate a corresponding support theorem for weighted spherical means.

We need the following result from \cite{RS}. Let $Z_{B,\infty}$ be a class of
continuous functions on $\Ann(B,\infty)=\{z\in\mathbb C^n: B<|z|<\infty\}$ such
that $f\times\mu_r(z)=0$ for all $z\in\mathbb C^n$ and $r>|z|+B.$

\begin{theorem}\label{Bth1}\emph{\cite{RS}}
A necessary and sufficient condition for $f\in Z_{B,\infty}(\mathbb
C^n)$ is that for all $p,q \in \mathbb Z_+,$ $ 1\leq j \leq d(p,q),$
the spherical harmonic coefficients $\tilde{a}_j^{p,q}$ of $f$
satisfy the following conditions:
\begin{enumerate}
\item For $p=0, q=0$ and $r<\rho<R,$
$ \tilde{a}^{0}(\rho)=0.$

\item For $p,q \geq 1$ and $r<\rho<R,$ there exists $c_i, d_k \in
\mathbb C$ such that

\[ \tilde{a}_j^{p,q}(\rho)=\sum_{i=1}^{p} c_i e^{\frac{1}{4}\rho^2}\rho^{-2(n+p+q-i)}+
\sum_{k=1}^{q} d_k e^{-\frac{1}{4}\rho^2}\rho^{-2(n+p+q-k)}.\]

\item For $q=0$ and $p\geq 1$ or $p=0$ and $q\geq 1$ and
$r<\rho<R,$ there exists $c_i, d_k \in \mathbb C$ such that
\[\tilde{a}_j^{p,0}(\rho)=\sum_{i=1}^{p} c_{i} e^{\frac{1}{4}\rho^2}\rho^{-2(n+p-i)},
\tilde{a}_j^{0,q}(\rho)=\sum_{k=1}^{q} d_{k} e^{-\frac{1}{4}\rho^2}\rho^{-2(n+q-k)}.\]
\end{enumerate}
\end{theorem}
Since the Heisenberg group $H^n$ is non-commutative, the twisted spherical means $f\times\mu_r$
and $\mu_r\times f$ are not equal, in general. Using this fact, we have proved the following
support theorem which do not require any decay condition.

\begin{theorem} \label{Cth11}\emph{\cite{RS}}
 Let $f$ be a continuous function on $\mathbb C.$ Then $f$ is supported in $|z|\leq B$ if and
only if $f\times\mu_r=\mu_r\times f=0$  for $s>B+|z|$ and $\forall z\in\mathbb C.$
\end{theorem}

We shall need the following lemmas in the proof of Theorem \ref{Cth3}.

\begin{lemma}\label{lemma6C}
Let $d\nu^{p,q}_\rho=P_1d\mu_\rho$. Let $f$ be a smooth function
on $\mathbb C^n$ such that $f\times\nu^{p,q}_{\rho}(z)=0,$ for all $ z\in\mathbb C^n$
and for all $\rho>|z|+B$. Then $P_1(\widetilde Z)f\times\mu_\rho(z)=0,$
for all $z\in\mathbb C^n$ and for all $\rho>|z|+B.$ Equivalently,
$P_1(\widetilde Z)f\in Z_{B,\infty}(\mathbb C^n)$.
\end{lemma}
\begin{proof}
We first prove
\begin{equation}\label{Cexp8}
\widetilde Z_1^*f\times\nu_\rho^{p-1,q}(z)=0,
~z\in\mathbb C^n \text{ for } \rho>|z|+B.
\end{equation}
Let
$\partial_{\bar w_1}=2\dfrac{\partial}{\partial\bar{w_1}}=
\dfrac{\partial}{\partial\xi_1}+i\dfrac{\partial}{\partial\eta_1}, ~w_1=\xi_1+i\eta_1.$
Then
\begin{align*}
\int_{\Ann(r,\rho)}\partial_{\bar w_1}\left(f(z-w)e^{\frac{i}{2}\text{Im}(z.\bar{w})}w_1^{p-1}\bar{w_2}^q\right)dw\\
=\int_{|w|=\rho}f(z-w)e^{\frac{i}{2}\text{Im}(z.\bar{w})}w_1^{p-1}\bar{w_2}^q\frac{w_1}{\rho}d\mu_{\rho}(w)\\
-\int_{|w|=r}f(z-w)e^{-\frac{i}{2}\text{Im}(z.\bar{w})}w_1^{p-1}\bar{w_2}^q\frac{w_1}{r}d\mu_{r}(w)
=0.
\end{align*}
Thus we have the following equation
\[\int_{\Ann(r,\rho)}\partial_{\bar w_1}\left(f(z-w)e^{\frac{i}{2}\text{Im}(z.\bar{w})}w_1^{p-1}\bar{w_2}^q\right)dw=0.\]
Rewriting this equation in the polar form,
 we get
\[\int_{s=r}^{\rho} \int_{|w|=s}\partial_{\bar w_1}
\left(f(z-w)e^{\frac{i}{2}\text{Im}(z.\bar{w})}w_1^{p-1}\bar{w_2}^q\right)d\mu_s(w)~s^{2n-1}ds=0.
 \]
Differentiating the above equation with respect to $\rho$, we have
\[\int_{|w|=
\rho}\partial_{\bar w_1}\left(f(z-w)e^{\frac{i}{2}\text{Im}(z.\bar{w})}w_1^{p-1}\bar{w_2}^q\right)d\mu_{\rho}(w)=0,
\]
whenever $z\in \mathbb C^n$ and $\rho> |z|+B$. Computing the differential
inside the integral and rearranging the terms,  we get
\[\int_{|w|=\rho}\left(-\frac{\partial }{\partial\bar w_1}f(z-w)+
\frac{1}{4}z_1f(z-w) \right)e^{\frac{i}{2}\text{Im}(z.\bar{w})}w_1^{p-1}\bar{w_2}^qd\mu_{\rho}(w)=0.
\]
That is
\[\int_{|w|=\rho}\left(\frac{\partial}{\partial\bar z_1}f(z-w)+
\frac{1}{4}(z_1-w_1)f(z-w) \right)e^{\frac{i}{2}\text{Im}(z.\bar{w})}w_1^{p-1}\bar{w_2}^qd\mu_{\rho}(w)=0,
\]
which is  the Equation (\ref{Cexp8}). Proceeding in a similar way, it can be shown that
$P_1(\widetilde Z)f\times\mu_\rho(z)=0,$ whenever $z\in\mathbb C^n $ and $\rho>|z|+B.$
\end{proof}
As before, it is enough to prove Theorem \ref{Cth3} for the function of
type $\tilde a(\rho)P_{s,t}(z)$. We can see this in the following lemma.
\begin{lemma}\label{lemma8C}
Fix $p,q \in \mathbb Z^+$ and let $f\times\nu_{r,j}(z)=0,$ for all
$z\in \mathbb C^n$ and $ r>|z|+B $ and for all $j,1\leq j\leq
d(p,q)$. Then $f_{lm}\times\nu_{r,j}(z)=0,$ for all $z\in\mathbb
C^n$ and for all $\rho>|z|+B.$
\end{lemma}
\begin{proof}
The proof of this lemma is similar to the proof of Lemma \ref{lemma7C} and hence omitted.
\end{proof}

To prove Theorem \ref{Cth3}, in view of Lemma \ref{lemma8C}, it is enough to
prove the following result.
\begin{proposition} \label{Cprop7}
Let $f(z)=\tilde{a}P_{s,t}$ be a smooth function on $\mathbb C^n$
such that $|f(z)||z|^{k}\leq C_k~e^{-\frac{1}{4}|z|^2}, k\in\mathbb Z_+$.
 Let $f\times\nu_r(z)=0$, for all
$z\in \mathbb C^n$ and $ r>|z|+B $ and for all $j,~1\leq j\leq d(p,q)$.
Then $f=0$ whenever $|z|>B$.
\end{proposition}
\begin{proof}
We first prove the result in case when $p=1,q=0$. The argument for general
$p,q$ is very similar. In this case, by Lemma \ref{lemma6C}, we have
$\widetilde Z_1^*f\in Z_{B,\infty}(\mathbb C^n).$ Since $f=\tilde{a}P_{s,t}$,
a similar calculation as in \cite{RS}, p.2516-17, gives that
\[\widetilde Z_1^*f=\frac{{1}}{2\rho}\widetilde D^*\tilde{a}P_{s+1,t}+
\left\{\left(\frac{{1}}{2(\gamma-1)}\rho\widetilde D^*+1\right)\tilde{a}\right\}
\frac{\partial P_{s,t}}{\partial\bar{z_1}},\]
where $\gamma=n+s+t$. Since $\tilde A_1f\in Z_{B,\infty}(\mathbb C^n)$,
by Lemma \ref{lemma8C} and Theorem \ref{Bth1}, it follows that
\[\left(\frac{{1}}{2(\gamma-1)}\rho\widetilde D^*+1\right)\tilde{a}=
\sum_{i=1}^{s}~c'_i~e^{\frac{1}{4}\rho^2}\rho^{-2(\gamma-1-i)}+
\sum_{k=1}^{t-1}~d'_k~e^{-\frac{1}{4}\rho^2}\rho^{-2(\gamma-1-k)}
\]
and
\[\frac{{1}}{2\rho}\widetilde D^*\tilde{a}
=\sum_{i=1}^{s+1}c_ie^{\frac{1}{4}\rho^2}\rho^{-2(\gamma+1-i)}+
\sum_{k=1}^{t}d_ke^{-\frac{1}{4}\rho^2}\rho^{-2(\gamma+1-k)}\]
Solving these equations for $\tilde a$ we get
\[ \tilde{a}(\rho)=\sum_{i=1}^{s+1}~C_i~e^{\frac{1}{4}\rho^2}\rho^{-2(\gamma-i)}+
\sum_{k=1}^{t}~D_k~e^{-\frac{1}{4}\rho^2}\rho^{-2(\gamma-k)},~C_i,~D_k\in\mathbb C.\]
But the given decay condition on the function $f$ then implies that $\tilde{a}(\rho)=0$,
whenever $\rho>B$. Hence $f=0$ for $\rho>B$. For the weight $z_1^p\bar z_2^q$,
the computations are similar and therefore omitted.
\end{proof}

Next we take up the case of Euclidean weighed spherical means. We prove the following
lemma which is key to the proof of Theorem \ref{Cth4}. As in \cite{EK}, let
\[f_{lm}(x)=d_s\int_{SO(n)} f(\tau^{-1}x)t^{lm}_{\pi_s}(\tau)d\tau,\]
for any $l,m$ with $1\leq l,m\leq d_s.$

\begin{lemma}\label{lemma10C}
Let $f\ast\mu_{\rho,j}^k(x)=0,$ for all $x\in\mathbb R^n,~ \rho>|x|+B $
and for all $j,1\leq j\leq d_k$. Then $f_{lm}\ast\mu_{\rho,j}^k(x)=0,$
for all $x\in\mathbb R^n$ and for all $\rho>|x|+B.$
\end{lemma}
\begin{proof}
Since space $H_k$ is $SO(n)$-invariant by change of variables, it follows that
\[f_{lm}\ast\mu_{\rho,j}^k(x)=d_s\int_{SO(n)}t^{lm}_{\pi_s}(\tau)
\int_{S_\rho(o)}f(\tau^{-1}x+y)P_{kj}(\tau y)d\mu_\rho(w)d\tau=0,\]
whenever $x\in \mathbb R^n$ and $\rho>|x|+B$.
\end{proof}
For $x=(x_1,x_2,x_3\ldots,x_n)\in\mathbb R^n,$ we realize the function
$f(x_1,x_2,x_3,\ldots,x_n)$ as $f(x_1+ix_2,x_3,\ldots,x_n)$. Let $z_1=x_1+ix_2.$
Then we can write \[\partial_{\bar z_1}=2\dfrac{\partial}{\partial\bar z_1}=
\dfrac{\partial}{\partial x_1}+i\dfrac{\partial}{\partial x_2}.\]

We need the following result from \cite{EK}.
Let $Z_{B,\infty}$ be a class of continuous functions on
 $\Ann(B,\infty)=\{x\in\mathbb R^n: B<|x|<\infty\}$
such that $f\ast\mu_r(x)=0$ for all $x\in\mathbb R^n$ and $r>|x|+B.$

\begin{theorem}\label{Bth5}\emph{\cite{EK}}
A necessary and sufficient condition for $f\in Z_{B,\infty}(\mathbb R^n)$ is that
for all $k \in \mathbb Z_+,$ the spherical harmonic coefficients $a_{kj}$ of $f$
satisfy the following conditions.
\[a_{kj}(\rho) = \sum_{i=0}^{k-1}\alpha_{kj}^{i}\rho^{k-d-2i},\quad \alpha_{kj}^{i} \in\mathbb C,\]
for all $k > 0,$ $1 \leq j \leq d_k,$ and $a_0(\rho) = 0$ whenever $r<\rho<R.$
\end{theorem}

\begin{lemma}\label{lemma2C}
Let $P_k(x)=(x_1+ix_2)^k$. Suppose $f\ast\mu_{\rho}^k(x)=0,$ for all $ x\in\mathbb R^n$
and for all $\rho>|x|+B$. Then ~$\partial_{\bar z_1}^k f\ast\mu_\rho(x)=0$~ for all
$x\in\mathbb R^n$ and for all $\rho>|x|+B.$ Equivalently,
$\partial_{\bar z_1}^k f\in Z_{B,\infty}(\mathbb R^n)$.
\end{lemma}
\begin{proof}
We first prove
\begin{equation}\label{Cexp5}
\partial_{\bar z_1}f\ast\mu_\rho^{k-1}(x)=0,~x\in\mathbb R^n \text{ for } \rho>|x|+B.
\end{equation}
Let
$ {\partial}_{\bar w_1}=2\dfrac{\partial}{\partial\bar{w_1}}=
\dfrac{\partial}{\partial y_1}+i\dfrac{\partial}{\partial y_2}, ~w_1=y_1+iy_2.$
Then
\begin{align*}
\int_{\Ann(r,\rho)}\partial_{\bar w_1}\left(f(z_1+w_1,x_3+y_3,\ldots,x_n+y_n)w_1^{k-1}\right)dy\\
=\int_{|y|=\rho}f(z_1+w_1,x_3+y_3,\ldots,x_n+y_n)w_1^{k-1}\frac{w_1}{\rho}d\mu_{\rho}(y)\\
-\int_{|y|=r}f(z_1+w_1,x_3+y_3,\ldots,x_n+y_n)w_1^{k-1}\frac{w_1}{r}d\mu_{r}(w)
=0.
\end{align*}
Thus we have the following equation
\[\int_{\Ann(r,\rho)}\partial_{\bar w_1}\left(f(z_1+w_1,x_3+y_3,\ldots,x_n+y_n)w_1^{k-1}\right)dy=0.\]
Rewriting this equation into polar form,
we get
\[\int_{s=r}^{\rho} \int_{|y|=s}
\partial_{\bar w_1}\left(f(z_1+w_1,x_3+y_3,\ldots,x_n+y_n)w_1^{k-1}\right)d\mu_s(y)~s^{n-1}ds=0.
\]
Differentiating the above equation with respect to $\rho$, we have
\[\int_{|y|=
\rho}\partial_{\bar w_1}f\left(f(z_1+w_1,x_3+y_3,\ldots,x_n+y_n)w_1^{k-1}\right)d\mu_{\rho}(w)=0,
\]
whenever $x\in \mathbb R^n$ and $\rho> |x|+B$. Computing the differential
inside integral, we obtain the Equation (\ref{Cexp5}). Proceeding in a similar way,
it can be shown that $\partial_{\bar z_1}^k f\ast\mu_\rho^k(x)=0,$
whenever $x\in\mathbb R^n $ and $\rho>|x|+B.$
\end{proof}

To prove Theorem \ref{Cth4}, in view of Lemma \ref{lemma10C}, it is enough to
prove the following result.
\begin{proposition}\label{Cprop2}
Let $f(x)=\tilde a(|x|)P_s(x)\in C^\infty(\mathbb R^n)$ such that $|x|^mf(x)$ is bounded
for each $m\in\mathbb Z_+$. Let $f\ast\mu_\rho^k(x)=0,$ for all $x\in\mathbb R^n,~\rho>|x|+B$.
Then $f=0$ whenever $|x|>B$.
\end{proposition}
\begin{proof}
First we find $\tilde{a}(\rho)$ for $k=1$. For this, by Lemma \ref{lemma2C}
we have $\bar\partial_1 f\in Z_{B,\infty}(\mathbb R^n)$. A computation
similar to that in \cite{EK}, p.445-6, we can write
\[
\frac{\partial f}{\partial x_j}=\frac{{1}}{\rho}\frac{\partial\tilde{a}}{\partial\rho}P_{s+1}^j +
\left\{\left(\frac{{1}}{n+2(s-1)}\rho\frac{\partial}{\partial\rho}+1\right)\tilde{a}\right\}\frac{\partial P_s}{\partial x_j},
\]
where $P_{s+1}^j\in H_{s+1}$. Therefore,
\[\bar\partial_1 f=\frac{{1}}{\rho}\frac{\partial\tilde{a}}{\partial\rho}P_{s+1}+
\left\{\left(\frac{{1}}{n+2(s-1)}\rho\frac{\partial}{\partial\rho}+1\right)\tilde{a}\right\}\bar\partial_1P_s
\]
for some $P_{s+1}\in H_{s+1}$. By Lemmas $[\ref{lemma10C},\ref{lemma2C}]$
and Theorem \ref{Bth5}, it follows that
\[
\frac{{1}}{\rho}\frac{\partial\tilde{a}}{\partial\rho}=\sum_{i=0}^{s}~c_i\rho^{-n-2i}\]
and
\[\left(\frac{{1}}{n+2(s-1)}\rho\frac{\partial}{\partial\rho}+1\right)\tilde{a}=\sum_{i=0}^{s-2}d_i\rho^{-n-2i},
\]
where $c_i, d_i\in \mathbb C$.
Solving these equations for $\tilde a$ we get
\[
\tilde{a}(\rho)=\sum_{i=-1}^{s-1}~c_i'\rho^{-n-2i}, ~c_i'\in\mathbb C.
\]
The given decay condition on the function $f$ then implies that
$\tilde{a}(\rho)=0$, whenever $\rho>B$. Hence $f=0$ for $\rho>B$.
The case of general weight $(x_1+ix_2)^k$ follows from induction.
This completes the proof.
\end{proof}

\noindent{\bf Acknowledgements:} The author wishes to thank Rama Rawat for several
fruitful discussions during preparation of this article. The author would also like
to gratefully acknowledge the support provided by the Department of Atomic Energy,
government of India.


\begin{thebibliography}{11}
\bibitem{ABK} M. L. Agranovsky, C Berenstein and Kuchment, P. Peter,
{\em Approximation by spherical waves in $L\sp p$-spaces,} J. Geom. Anal., 6 (1996),  no. 3, 365--383.

\bibitem{AR} M. L. Agranovsky and R. Rawat, {\em Injectivity sets for spherical means on the Heisenberg group,}
J. Fourier Anal. Appl., 5 (1999), no. 4, 363--372.

\bibitem{EK} C. L. Epstein and B. Kleiner, {\em Spherical means in annular regions,}
 Comm. Pure Appl. Math., 46 (1993), no. 3, 441-451.

\bibitem{FL} M, Filaseta and  T-Y, Lam, {\em On the irreducibility of the
generalized Laguerre polynomials,} Acta Arith.  105 (2002),  no. 2, 177--182.

\bibitem{Ge} D. Geller, {\em Spherical harmonics, the Weyl transform and the Fourier transform on the Heisenberg group,}
Canad. J. Math.  36  (1984),  no. 4, 615--684.

\bibitem{H} S. Helgason, {\em The Radon Transform,} Birkhauser, 1983.

\bibitem{NRR} E. K. Narayanan, R. Rawat and S. K.  Ray, {\em Approximation by $K$-finite functions in $L\sp p$ spaces,}
Israel J. Math.  161  (2007), 187--207.

\bibitem{NT2} E. K. Narayanan and S. Thangavelu, {\em A spectral Paley-Wiener theorem for the Heisenberg group and a
support theorem for the twisted spherical means on ~$\mathbb C^n$}, Ann. Inst. Fourier (Grenoble),
56 (2006), no. 2, 459-473.

\bibitem{NT} E. K. Narayanan and S. Thangavelu, {\em Injectivity sets for spherical means on the Heisenberg group,}
 J. Math. Anal. Appl. 263 (2001), no. 2, 565-579.

\bibitem{Q1} E. T. Quinto, {\em Helgason's support theorem and spherical Radon transforms,}
Contemp. Math., 464(2008), 249-264.

\bibitem{Q2} E. T. Quinto, {\em Real analytic Radon transforms on rank one symmetric spaces,}
 Proc. Amer. Math. Soc. 117 (1993), no. 1, 179–186.

\bibitem{Q3} J. Boman and E. T. Quinto, {\em Support theorems for real analytic Radon transforms,}
 Duke Math. J. 55 (1987), 943-948.

\bibitem{RS} R. Rawat and R. K. Srivastava, {\em Twisted spherical means in annular regions in
$\mathbb C^n$ and support theorems, } Ann. Inst. Fourier (Grenoble), 59 (2009), no. 6, 2509-2523.

\bibitem{Ru} W. Rudin, {\em Function theory in the unit ball of ~$\mathbb C^n$,} Springer-Verlag,
 New York-Berlin, 1980.

\bibitem{T} S. Thangavelu, {\em An introduction to the uncertainty
principle}, Prog. Math. 217, Birkhauser, Boston (2004).
\end{thebibliography}
\end{document}